\documentclass[12pt,a4paper]{amsart}
\usepackage{amsfonts}
\usepackage{amsthm}
\usepackage{amsmath, amscd, amssymb, mathtools, color}
\usepackage[latin2]{inputenc}
\usepackage{t1enc}
\usepackage[mathscr]{eucal}
\usepackage{indentfirst}
\usepackage{graphicx}
\usepackage{graphics, tikz}
\usepackage{hyperref}
\usepackage{pict2e}
\usepackage{epic}
\numberwithin{equation}{section}
\usepackage[margin=2.9cm]{geometry}
\usepackage{epstopdf}

\DeclareMathOperator{\reg}{reg}
\DeclareMathOperator{\lex}{lex}

\def\NZQ{\Bbb}               

\def\ZZ{{\NZQ Z}}

\def\frk{\frak}               

\def\Phi{{\frk n}}
\def\Phi{{\frk N}}
\def\MR{{\mathcal R}}

\def\MA{{\mathcal A}}

\def\MC{{\mathcal C}}
\def\MB{{\mathcal B}}
\def\MD{{\mathcal D}}

\def\ab{{\bold a}}
\def\bb{{\bold b}}

\def\opn#1#2{\def#1{\operatorname{#2}}} 
\opn\chara{char} \opn\length{\ell} \opn\pd{pd} \opn\rk{rk}
\opn\projdim{proj\,dim} \opn\injdim{inj\,dim} \opn\rank{rank}
\opn\depth{depth} \opn\grade{grade} \opn\height{height}
\opn\embdim{emb\,dim} \opn\codim{codim}

\opn\Tr{Tr} \opn\bigrank{big\,rank}
\opn\superheight{superheight}\opn\lcm{lcm}
\opn\trdeg{tr\,deg}
\opn\reg{reg} \opn\lreg{lreg} \opn\ini{in} \opn\lpd{lpd}
\opn\size{size}\opn\bigsize{bigsize}
\opn\cosize{cosize}\opn\bigcosize{bigcosize}
\opn\sdepth{sdepth}\opn\sreg{sreg}
\opn\link{link}\opn\fdepth{fdepth}
\opn\div{div} \opn\Div{Div} \opn\cl{cl} \opn\Cl{Cl}
\let\epsilon\varepsilon
\let\phi=\varphi
\let\kappa=\varkappa
\opn\Spec{Spec} \opn\Supp{Supp} \opn\supp{supp} \opn\Sing{Sing}
\opn\Ass{Ass} \opn\Min{Min}\opn\Mon{Mon} \opn\dstab{dstab} \opn\astab{astab}
\opn\Syz{Syz}
\opn\Ann{Ann} \opn\Rad{Rad} \opn\Soc{Soc}
\let\to=\rightarrow

\def\Implies{\ifmmode\Longrightarrow \else
        \unskip${}\Longrightarrow{}$\ignorespaces\fi}
\def\implies{\ifmmode\Rightarrow \else
        \unskip${}\Rightarrow{}$\ignorespaces\fi}
\def\iff{\ifmmode\Longleftrightarrow \else
        \unskip${}\Longleftrightarrow{}$\ignorespaces\fi}

\let\:=\colon
 \theoremstyle{plain}
\newtheorem{Theorem}{Theorem}[section]
 \newtheorem{Lemma}[Theorem]{Lemma}
 \newtheorem{Corollary}[Theorem]{Corollary}
 \newtheorem{Proposition}[Theorem]{Proposition}
 
 \newtheorem{conjecture}[Theorem]{Conjecture}
 \newtheorem{question}[Theorem]{Question}
 \theoremstyle{definition}
 \newtheorem{Definition}[Theorem]{Definition}
 \newtheorem{Remark}[Theorem]{Remark}
 
 \newtheorem{Example}[Theorem]{Example}
 
 \newtheorem{Notation}[Theorem]{Notation}
 
\let\epsilon\varepsilon
\let\kappa=\varkappa
\begin{document}

\title{Second Powers of Cover Ideals of Paths}

\author[N. Erey]{Nursel Erey}

\address{Gebze Technical University \\ Department of Mathematics \\
Gebze \\ Kocaeli \\ 41400 \\ Turkey} 

\email{nurselerey@gtu.edu.tr}

\author[A. A. Qureshi]{Ayesha Asloob Qureshi}

\address{Sabanc\i \; University, Faculty of Engineering and Natural Sciences, Orta Mahalle, Tuzla 34956, Istanbul, Turkey} 

\email{aqureshi@sabanciuniv.edu}



\begin{abstract} 
We show that the second power of the cover ideal of a path graph has linear quotients. To prove our result we construct a recursively defined order on the generators of the ideal which yields linear quotients. Our construction has a natural generalization to the larger class of chordal graphs. This generalization allows us to raise some questions that are related to some open problems about powers of cover ideals of chordal graphs.  
\end{abstract}

\maketitle


\section{Introduction}
Let $S=K[x_1,\dots ,x_n]$ be the polynomial ring in $n$ variables over a field $K$. We say that a monomial ideal $I$ has \emph{linear quotients} if there exists an order $u_1,\dots ,u_r$ of its minimal monomial generators such that for each $i=2,\dots ,r$ there exists a subset of $\{x_1,\dots ,x_n\}$ which generates the colon ideal $(u_1,\dots ,u_{i-1}):(u_i)$.

Ideals with linear quotients were introduced in \cite{ht}. Many interesting classes of ideals are known to have linear quotients. For example, stable ideals, squarefree stable ideals and (weakly) polymatroidal ideals all have linear quotients. Moreover, in the squarefree case having linear quotients translates into the shellability concept in combinatorial topology. Indeed, if $I$ is the Stanley-Reisner ideal of a simplicial complex $\Delta$, then $I$ has linear quotients if and only if the Alexander dual of $\Delta$ is shellable.

If an ideal $I$ has linear quotients, then $I$ is componentwise linear, i.e., for each $d$, the ideal generated by all degree $d$ elements of $I$ has a linear resolution. In particular, if $I$ is generated in single degree and has linear quotients, then it has a linear resolution. Herzog, Hibi and Zheng \cite{herzog hibi zheng} proved that when $I$ is a monomial ideal generated in degree $2$, the ideal $I$ has a linear resolution if and only if it has linear quotients. Moreover, they proved that if $I$ has a linear resolution, then so does every power of it.

Given a finite simple graph $G$ with vertices $x_1,\dots ,x_n$ the \emph{edge ideal} of $G$, denoted by $I(G)$, is generated by the monomials $x_ix_j$ such that $x_i$ and $x_j$ are adjacent vertices. Edge ideals are extensively studied in the literature, see for example survey papers \cite{ha survey, vil survey}. Since every edge ideal is generated in degree $2$, having linear resolution and having linear quotients are equivalent concepts for such ideals. Also, 
due to a result of Fr\"{o}berg \cite{froberg} it is known that the edge ideal $I(G)$ of a graph $G$ has a linear resolution if and only if the complement graph of $G$ is chordal.

The Alexander dual of $I(G)$ is known as the \emph{(vertex) cover ideal} of $G$ and, it is denoted by $J(G)$. Note that $J(G)$ is defined by
$$J(G)=(x_{i_1}\dots x_{i_k}: \{x_{i_1},\dots ,x_{i_k}\} \text{ is a minimal vertex cover of } G). $$

Due to a result of Herzog and Hibi \cite{herzog hibi componentwise} it is known that Stanley-Reisner ideal arising from a simplicial complex $\Delta$ is (componentwise) linear if and only if the Alexander dual of $\Delta$ is (sequentially) Cohen-Macaulay. This implies that
for any graph, being (sequentially) Cohen-Macaulay is equivalent to having a (componentwise) linear cover ideal.  Unlike the edge ideals, there is no combinatorial characterization of cover ideals with linear resolutions. In fact, the authors of \cite{fh} describe example of a graph in \cite[Example 4.4]{fh} whose  cover ideal has linear resolution if and only if the characteristic of the ground field is not two. The problem of classifying all Cohen-Macaulay or sequentially Cohen-Macaulay graphs is considered to be intractable and thus this problem is studied for special classes of graphs, for examples see \cite{estrada vil, fh, francisco van tuyl, herzog hibi dist, hhz chordal, vv, woodroofe chordal}.

Powers of edge ideals were studied by many authors recently, see for example the survey article \cite{banerjee survey}. A main motivation to study powers is to understand the behaviour of (Castelnuovo-Mumford) regularity in terms of graph properties. It is well-known \cite{cutkosky, kodiyalam} that if $I$ is a homogeneous ideal, then $\reg(I^s)$, the regularity of $I^s$, is a linear function in $s$ for sufficiently large $s$. Powers of cover ideals were relatively less explored than edge ideals in the literature. For example, in \cite{hang} for a unimodular hypergraph $\mathcal{H}$ (in particular bipartite graph) the regularity of $J(\mathcal{H})^s$ was determined for $s$ big enough. However, the regularity of $J(\mathcal{H})^s$ is unknown for small values of $s$. The reader can refer to \cite{constan, erey jpaa, fakhari reg, fakhari symbolic, hang, mohammadi2, selvar} for some recent articles where powers of cover ideals were studied.

Van Tuyl and Villarreal \cite{vv} showed that cover ideal of a chordal graph has linear quotients, extending the results in \cite{francisco van tuyl} where it was showed that such ideals are componentwise linear. In fact, Woodroofe \cite{woodroofe} showed that independence complex of a graph with no chordless cycles of length other than $3$ or $5$ is vertex decomposable and hence shellable. In \cite{herzog hibi ohsugi} the authors studied powers of componentwise linear ideals and they showed that all powers of a Cohen-Macaulay chordal graph have linear resolutions. Their proof was based on the method of \emph{$x$-condition} which, when satisfied, guarantees that all powers of the ideal have linear resolutions. More generally they proposed the following conjecture:

\begin{conjecture}\label{conjecture}\cite[Conjecture 2.5]{herzog hibi ohsugi} All powers of the vertex cover ideal of a chordal graph are componentwise linear.
\end{conjecture}

The $x$-condition method requires that the ideal is generated by the same degree monomials, which is indeed the case for the cover ideal of a Cohen-Macaulay chordal graph. However, generators of cover ideal of an arbitrary chordal graph can have different degrees. Therefore, the $x$-condition method cannot be applied in the general case. There has not been any progress on Conjecture~\ref{conjecture} except very few classes of graphs. In addition to Cohen-Macaulay chordal graphs it is known that the conjecture holds for generalized star graphs \cite{mohammadi powers of chordal}. Also, powers of cover ideals of Cohen-Macaulay chordal graphs are known to have linear quotients \cite{mohammadi powers of chordal}.

It is unknown if Conjecture~\ref{conjecture} is true for the second power of the cover ideal of a chordal graph. The following question arises naturally:
\begin{question}\label{myquestion} Let $G$ be a chordal graph.
	\begin{enumerate}
		\item Does $J(G)^2$ have linear quotients?
		\item Does $J(G)^s$ have linear quotients for all $s$?
	\end{enumerate}
\end{question}

In this paper we address Question~\ref{myquestion}~(1) for a path graph $P_n$. Our main result Theorem~\ref{main thm} states that second power of cover ideal of a path has linear quotients. We construct a recursively defined order, which we call \emph{rooted order}, on the minimal generators of $J(P_n)^2$ which produces linear quotients. Our method is purely combinatorial and it is completely different from the $x$-condition method which involves the study of Gr\"{o}bner bases of defining ideals of Rees algebras. 

We summarize the contents of this paper. In Section~\ref{sec:definitions} we introduce the necessary definitions and notations. Section~\ref{sec:properties of generators and rooted lists} is devoted to some technical results about the rooted lists as well as minimal generators of $J(P_n)$ and $J(P_n)^2$ which will be needed in the next section. In Section~\ref{sec:2fold products and minimal generators} we analyse some cases where the product of two generators of the cover ideal may not produce a minimal generator for the second power of the ideal. Note that if a monomial ideal $I$ is generated in the same degree and, $u$ and $v$ are two minimal generators, then the $2$-fold product $uv$ is necessarily a minimal generator of $I^2$. Since the cover ideal of a path is generated in different degrees, describing generators of the second power of the cover ideal is not trivial as in the case of equigenerated ideals. The goal of Section~\ref{sec:main results} is to prove the main result that $J(P_n)^2$ has linear quotients. We also extend the concept of rooted order to chordal graphs and discuss Question~\ref{myquestion}. 

\section{Definitions and Notations}\label{sec:definitions}

Let $K$ be a field and $S=K[x_1, x_2, \ldots, x_n]$ be the polynomial ring over $K$ in $n$ indeterminates. Let $G$ be a finite simple graph with the vertex set $V(G)=\{x_1, x_2, \ldots, x_n\}$ and the edge set $E(G)$. Then the \emph{edge ideal} $I(G) \subset S$ of $G$ is generated by all quadratic monomials $x_ix_j$ such that $\{x_i,x_j\} \in E(G)$. A {\em vertex cover} $C$ of $G$ is  a subset of $V(G)$ such that $C \cap e \neq \emptyset$, for all $e \in E(G)$. A vertex cover of $G$ is called {\em minimal} if it is not strictly contained in any other vertex cover of $G$. Let $\mathcal{M}(G)$ be the set of all minimal vertex covers of $G$. Then the \emph{(vertex) cover ideal} of $G$, denoted by $J(G)$ is generated by $x_{i_1}x_{i_2}\cdots x_{i_k}$ such that $\{x_{i_1},x_{i_2},\ldots ,x_{i_k}\} \in \mathcal{M}(G)$. It is a well-known fact that $J(G)$ is the Alexander dual of $I(G)$. Throughout this paper we will use a set of vertices $C$ interchangeably with its corresponding monomial $\prod_{x_i\in C}x_i$.

A graph is called \emph{chordal} if it has no induced cycles except triangles. Every chordal graph contains a \emph{simplicial vertex}, i.e., a vertex whose neighbors form a complete graph. The graph $G$ is called a {\em path} on $\{x_1, x_2, \ldots, x_n\}$ if $$E(G)=\{\{x_1,x_2\}, \{x_2,x_3\}, \ldots, \{x_{n-1},x_n\}\}.$$ We denote a path on $n$ vertices by $P_n$. Note that every path is a chordal graph. Our main goal is to prove that $J(P_n)^2$ has linear quotients. If $I$ is a monomial ideal, we denote by $G(I)$ the set of the minimal monomial generators of $I$. Recall that a monomial ideal $I$ is said to have {\em linear quotients} if  there exists a suitable order of the minimal generators $u_1,u_2, \ldots, u_m$ such that for all $2 \leq i \leq m$ the ideal $(u_1, \ldots, u_{i-1}):(u_i)$ is generated by variables. Given two monomials $u$ and $v$, we will use the notation $u:v$ for the monomial $u/\gcd(u,v)$.

 To simplify the notation in the following text, we set $uA:=\{ua: a \in A\}$, where $u$ is a  monomial in $ S$ and $A$ is a subset of $ S$. Similarly, if $A=a_1, \ldots, a_n$ is a list of elements of $S$, then $uA$ is a new list defined by $uA:=ua_1, \ldots, ua_n$. Note that we use a non-standard way to represent a list. Normally, one would write $A=(a_1, \ldots, a_n)$ but we will drop the parentheses to avoid possible confusion between lists and ideals. The following lemma gives the relation between $\mathcal{M}(P_n)$, $\mathcal{M}(P_{n-2})$ and $\mathcal{M}(P_{n-3})$, or equivalently, the minimal generators of $J(P_n), J(P_{n-2})$ and $J(P_{n-3})$.

\begin{Lemma}\label{lem: rooted order for first power}
	For all $n \geq 5$, we have \[
	 G(J(P_{n}))=x_{n-1}G(J(P_{n-2})) \sqcup\; x_nx_{n-2}G(J(P_{n-3})).\]
	 Moreover, if $u_1,\dots ,u_p$ and $v_1,\dots, v_q$ are the minimal generators of $J(P_{n-2})$ and $J(P_{n-3})$ respectively written in linear quotients order, then $J(P_n)$ has linear quotients with respect to the order $x_{n-1}u_1,\dots ,x_{n-1}u_p,  x_nx_{n-2}v_1,\dots,  x_nx_{n-2}v_q $.
\end{Lemma}
\begin{proof}
	Since $P_n$ is a chordal graph and $x_n$ is a simplicial vertex, the result follows from \cite[Theorem 3.1]{erey}.
\end{proof}

Based on Lemma~\ref{lem: rooted order for first power}, we define a recursive order on the generators of $J(P_n)$.
\begin{Definition}[\textbf{Rooted list, rooted order}]\label{def:rooted list} 

Let $P_n$ be the path with edge ideal $I(P_n)=(x_1x_2, x_2x_3,\dots ,x_{n-1}x_n)$.  
We define the {\em rooted list}, denoted by  $\MR(P_n)$, of minimal generators of $J(P_n)$ recursively as follows:
\begin{itemize}
\item  $\MR(P_2)=x_1, x_2$

\item  $\MR(P_3)=x_2,x_1x_3$

\item $\MR(P_4)=x_1x_3,x_2x_3, x_2x_4$

\item for $n\geq 5$, if $\MR(P_{n-2})=u_1,\dots , u_r$ and $\MR(P_{n-3})=v_1,\dots , v_s$ then
	\begin{equation*}\label{eq:rooted list}
	\MR(P_n)=x_{n-1}u_1,\dots ,x_{n-1}u_r, x_nx_{n-2}v_1,\dots ,x_nx_{n-2}v_s.
	\end{equation*}
\end{itemize}
We set $\MR(P_1)$ as an empty list. Moreover, we define a total order $>_\MR$ which we call a \emph{rooted order} of the minimal generators of $J(P_n)$ as follows: if $\MR(P_n)=w_1, \ldots , w_t$, then $w_i >_\MR w_j$ for $i<j$.
\end{Definition}

\begin{figure}[hbt!]
	\includegraphics[width = 15cm]{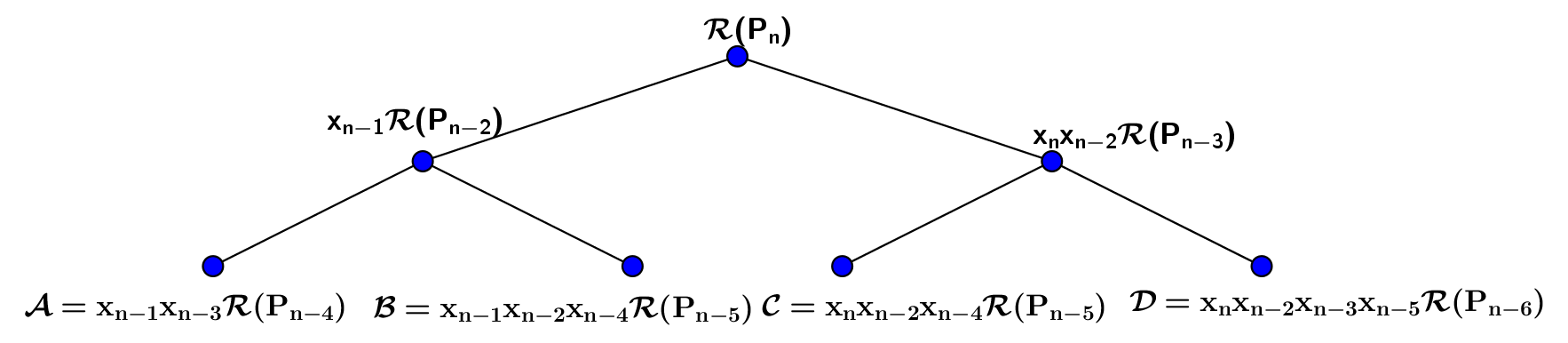}
	\caption{Branching of $\MR(P_n)$}
	\label{branching}
\end{figure}

\begin{Remark}\label{rk:rooted order linear quotients}
	Let $\MR(P_n)=u_1,\dots ,u_q$ for $n\geq 2$. Then Lemma~\ref{lem: rooted order for first power} together with definition of rooted list implies that $J(P_n)$ has linear quotients with respect to the order $u_1,\dots ,u_q$.
\end{Remark}

Let $\mathbf{a}=(a_1, \ldots ,a_n)$, $\mathbf{b}=(b_1, \ldots, b_n)$ be two elements in $\ZZ^n$. Then we say that $\mathbf{a}>_{\lex} \mathbf{b}$ if first non-zero entry in $\ab-\bb$ is positive. In the following definition we adopt the same terminology used in Discussion 4.1 in \cite{banerjee}.

\begin{Definition}[\textbf{$\mathbf{2}$-fold product, maximal expression}]
Let $I=(u_1,\dots ,u_q)$.	We say that $M=u_1^{a_1}\dots u_q^{a_q}$ is a {\em $2$-fold product} of minimal generators of $I$ if $a_i\geq 0$ and $a_1+\dots +a_q=2$. We write $u_1^{a_1}\dots u_q^{a_q}>_{lex} u_1^{b_1}\dots u_q^{b_q}$ if $(a_1,\dots ,a_q)>_{lex} (b_1,\dots ,b_q)$.
We say that $M=u_1^{a_1}\dots u_q^{a_q}$ is a {\em maximal expression} if $(a_1,\dots ,a_q)>_{lex} (b_1,\dots ,b_q)$ for any other $2$-fold product $M=u_1^{b_1}\dots u_q^{b_q}$. 
\end{Definition}

\begin{Notation}
For a monomial ideal $I$ we set $F(I^2)=\{uv : u, v \in G(I)\}$.
\end{Notation}

 Note that for an arbitrary monomial ideal $I$, while $G(I^2)\subseteq F(I^2)$, not every $2$-fold product is a minimal generator of $I^2$. However, if $I$ is generated by the same degree monomials, in particular if $I$ is an edge ideal, then $G(I^2)=F(I^2)$.


\begin{Definition}[\textbf{Rooted order on the second power}]\label{def:rooted for second power} Let $\MR (P_n)=u_1,\dots ,u_q$ for $n\geq 2$. We define a total order $>_{\MR}$ on $F(J(P_n)^2)$ which we call \emph{rooted order} as follows. For $M,N\in F(J(P_n)^2)$ with maximal expressions $M=u_1^{a_1}\dots u_q^{a_q}$ and $N=u_1^{b_1}\dots u_q^{b_q}$ we set $M>_{\MR} N$ if $(a_1,\dots ,a_q)>_{lex} (b_1,\dots ,b_q)$. 

Let $G(J(P_n)^2)=\{U_1, U_2, \ldots, U_s\}$. Then we say $U_1,U_2,\dots ,U_s$ is a \emph{rooted list} of generators of $J(P_n)^2$  if $U_1>_{\MR}>U_2>_{\MR}\ldots >_{\MR} U_s$. In such case, we denote the rooted list of generators by $\MR(J(P_n)^2)=U_1,U_2, \ldots, U_s.$ 
\end{Definition}


The following table shows the rooted list $\mathcal{R}(P_n)$ for $2 \leq n \leq 7$. \\
\begin{table}[ht]
\caption{}
\centering
\begin{tabular}{ |p{1cm}|p{12.8cm} | }

 \hline
$n$& $\mathcal{R}(P_n)$ \\

\hline
2 & $u_1=x_1,u_2=x_2$\\
\hline
3 & $u_1=x_2,u_2=x_1x_3$\\
\hline

4 &  $u_1=x_1x_3,u_2=x_2x_3, u_3=x_2x_4$\\

\hline
5& $u_1=x_2x_4,u_2=x_1x_3x_4, u_3=x_1x_3x_5, u_4=x_2x_3x_5$     \\
\hline

6 &  $u_1=x_1x_3x_5,u_2=x_2x_3x_5, u_3=x_2x_4x_5, u_4=x_2x_4x_6, u_5=x_1x_3x_4x_6$\\

\hline

7& $u_1=x_2x_4x_6,u_2=x_1x_3x_4x_6, u_3=x_1x_3x_5x_6, u_4=x_2x_3x_5x_6$, \\

&$u_5=x_1x_3x_5x_7,u_6=x_2x_3x_5x_7, u_7=x_2x_4x_5x_7$     \\
\hline

\end{tabular}
 \label{Table1}
\end{table}

\medskip
Given above labelling of elements of $\mathcal{R}(P_n)$, for $2 \leq n \leq 7$, Table~\ref{Table2} shows the rooted list of the minimal generators of $J(P_n)^2$, and the 2-fold products in $F(J(P_n)^2)\setminus G(J(P_n)^2)$. 
\\
\begin{table}[ht]
\caption{}
\centering
\begin{tabular}{ |p{0.4cm}|p{7cm}|p{6.1cm}|  }

 \hline
$n$& $\mathcal{R}(J(P_n)^2)$ & $F(J(P_n)^2)\setminus G(J(P_n)^2)$\\

 \hline
 2&  $u_1^2, u_1u_2, u_2^2$&  \\
 \hline
 3&  $u_1^2, u_1u_2, u_2^2$&  \\

 \hline
 4 &  $u_1^2, u_1u_2, u_1u_3, u_2^2, u_2u_3, u_3^2$ &  \\
 \hline
 
 5 &  $u_1^2, u_1u_2, u_1u_3, u_1u_4, u_2^2, u_2u_3, u_3^2,u_3u_4, u_4^2$ & $u_2u_4$ (divisible by $u_1u_3$) \\
 \hline
 
  6 &  $u_1^2, u_1u_2, u_1u_3, u_1u_4, u_1u_5, u_2^2, u_2u_3,$ & $ u_2u_5, u_3u_5$ (divisible by $u_1u_4$)\\

    &   $u_2u_4, u_3^2,u_3u_4, u_4^2, u_4u_5, u_5^2$ &  \\
 \hline
 
  7 &  $u_1^2, \, u_1u_2,\, u_1u_3, \, u_1u_4, \, u_1u_5, \,  u_1u_6,\,  u_1u_7, $ & $u_2u_4$ (divisible by $u_1u_3$) \\
  
     &  $ u_2^2, \, u_2u_3, \, u_2u_5, \, u_3^2, \, u_3u_4, \, u_3u_5,$ & $u_2u_6, u_2u_7, u_3u_7$ (divisible by $u_1u_5$) \\
       &  $u_3u_6=u_4u_5,$ & $u_4u_7$ (divisible by $u_1u_6$) \\
      & $  u_4^2, \, u_4u_6, \, u_5^2, \, u_5u_6, \, u_5u_7, \, u_6^2, \, u_6u_7, \, u_7^2$ & \\
  
 \hline
\end{tabular}
\label{Table2}
\end{table}


\section{Some properties of $G(J(P_n))$, $G(J(P_n)^2)$ and rooted lists}\label{sec:properties of generators and rooted lists}

In this section we will prove some technical results about properties of rooted lists. We start with some observations:

\begin{Remark}\label{rk:max expression preserved P_{n-2}}
	Let $n\geq 4$ and let $\MR(P_{n-2})=u_1,\dots ,u_m$. Observe that by definition of rooted order for any $k,\ell \in\{1,\dots ,m\}$ we have
	\[u_k>_{\MR} u_\ell \,\text{ in } \MR(P_{n-2}) \Longleftrightarrow x_{n-1}u_k>_{\MR} x_{n-1}u_\ell \, \text{ in } \MR(P_n).\]
	
	Therefore $u_iu_j$ is a maximal expression with $i\leq j$ if and only if $(x_{n-1}u_i)(x_{n-1}u_j)$ is a maximal expression with $x_{n-1}u_i \geq_{\MR} x_{n-1}u_j$.
\end{Remark}

\begin{Remark}\label{rk:max expression preserved P_{n-3}}
	Let $n\geq 5$ and let $\MR(P_{n-3})=u_1,\dots ,u_m$. Observe that by definition of rooted order for any $k,\ell \in\{1,\dots ,m\}$ we have
	\[u_k>_{\MR} u_\ell \,\text{ in } \MR(P_{n-3}) \Longleftrightarrow x_nx_{n-2}u_k>_{\MR} x_nx_{n-2}u_\ell \, \text{ in } \MR(P_n).\]
	
	Therefore $u_iu_j$ is a maximal expression with $i\leq j$ if and only if $(x_nx_{n-2}u_i)(x_nx_{n-2}u_j)$ is a maximal expression with $x_nx_{n-2}u_i \geq_{\MR} x_nx_{n-2}u_j$.
\end{Remark}

\begin{Lemma}\label{lem: order preserved for P_n-2}
	Let $n\geq 4$. Then $U,V\in F(J(P_{n-2})^2)$ if and only if $x_{n-1}^2U, x_{n-1}^2V\in F(J(P_n)^2)$. Moreover, in such case $U>_{\MR} V$ if and only if $x_{n-1}^2U>_{\MR} x_{n-1}^2V$.
\end{Lemma}

\begin{proof}
	The first statement is clear by Lemma~\ref{lem: rooted order for first power}.  Let $\MR(P_{n-2})=u_1,\dots ,u_m$. Suppose that $U=u_iu_j$ and $V=u_su_t$ are maximal expressions where $i\leq j$ and $s\leq t$. Then by Remark~\ref{rk:max expression preserved P_{n-2}} it follows that $x_{n-1}^2U=(x_{n-1}u_i)(x_{n-1}u_j)$ and $x_{n-1}^2V=(x_{n-1}u_s)(x_{n-1}u_t)$ are maximal expressions where $x_{n-1}u_i \geq_{\MR} x_{n-1}u_j$ and $x_{n-1}u_s \geq_{\MR} x_{n-1}u_t$. Keeping Remark~\ref{rk:max expression preserved P_{n-2}} in mind observe that
	
	\begin{equation*}
	\begin{split}
	U>_{\MR} V& \Longleftrightarrow \text{ either } u_i>_{\MR} u_s \text{ or } u_i=u_s \text{ and } u_j>_{\MR} u_t  \\
	& \Longleftrightarrow \text{ either } x_{n-1}u_i>_{\MR} x_{n-1}u_s \text{ or } x_{n-1}u_i=x_{n-1}u_s \text{ and } x_{n-1}u_j>_{\MR} x_{n-1}u_t  \\
	& \Longleftrightarrow x_{n-1}^2U>_{\MR} x_{n-1}^2V.
	\end{split}
	\end{equation*} 
\end{proof}

\begin{Lemma}\label{lem: order preserved for P_n-3}
	Let $n\geq 5$. Then $U,V\in F(J(P_{n-3})^2)$ if and only if both $x_n^2x_{n-2}^2U$ and $x_n^2x_{n-2}^2V$ belong to $F(J(P_n)^2)$. Moreover, in such case $U>_{\MR} V$ if and only if $x_n^2x_{n-2}^2U>_{\MR}x_n^2 x_{n-2}^2V$.
\end{Lemma}
\begin{proof}
	The proof is almost identical to that of previous lemma if one uses Remark~\ref{rk:max expression preserved P_{n-3}} instead of Remark~\ref{rk:max expression preserved P_{n-2}}.
\end{proof}

\begin{Lemma}\label{lem:every b contains some a}
	Let $n\geq 4$ and let $u \in G(J(P_n)) $ such that $x_n | u$. Then there exists $v \in  G(J(P_{n-2}))$ such that $v$ divides $u/x_n$.
\end{Lemma}

\begin{proof}
	If $x_n |u$, then $u$ is not divisible by $x_{n-1}$ because $u$ is a minimal vertex cover of $P_n$. Then $u/x_n$ contains a minimal vertex cover of $P_{n-2}$ which verifies the statement. 
\end{proof}
\begin{Lemma}\label{lem: colon contains the second variable from the end}
	Let $u_1,\dots ,u_r$ be the rooted list of $P_n$. Let $k$ be the smallest index such that $x_{n-1}\nmid u_k$. Then $(u_1,\dots ,u_{k-1}):(u_k)=(x_{n-1})$. Also if $i>k$, then
	$$(u_1,\dots ,u_{i-1}):(u_i)=(x_{n-1})+(u_k,\dots ,u_{i-1}):(u_i).  $$
\end{Lemma}
\begin{proof}
	The statement is clear when $n=2$ or $n=3$. Otherwise it follows from Lemma~\ref{lem:every b contains some a}.
\end{proof}
\begin{Lemma}\label{lem:properties of rooted list}

	Let $\MR(P_n)=u_1,\dots ,u_m$.
	\begin{enumerate}
	\item If $x_n\mid u_i$ for some $i$, then $x_n\mid u_j$ for all $j\geq i$.
		\item If $x_{n-2}\mid u_i$ for some $i$, then $x_{n-2}\mid u_j$ for all $j\geq i$.
		\item If $\MR(P_{n-2})=v_1,\dots ,v_k$ and $\MR(P_{n-1})=w_1,\dots ,w_{\ell}$, then
		$$\MR(P_n)=x_{n-1}v_1,\dots, x_{n-1}v_k,x_nw_1,\dots ,x_nw_\alpha $$
		for some $\alpha <\ell$.
	\end{enumerate}
\end{Lemma}

\begin{proof}
	$(1)$ follows from the definition of rooted list. 
	
	$(2)$ can be confirmed by applying $(1)$ to $P_{n-2}$ in the recursive definition of $\MR(P_n)$.
	
	To see $(3)$ we will just compare the recursively defined lists of $P_{n-1}$ and $P_n$. For $n\leq 5$, one can refer to Table~\ref{Table1} to confirm the statement holds. So, let us assume that $n\geq 6$. Let $\MR(P_{n-4})=y_1,\dots, y_s$ and let $\MR(P_{n-3})=z_1,\dots,z_t$. Then by recursive definition of rooted list we have
	\[\MR(P_n)=x_{n-1}v_1,\dots, x_{n-1}v_k, x_nx_{n-2}z_1,\dots ,x_nx_{n-2}z_t \] and 
	  	\[\MR(P_{n-1})=x_{n-2}z_1,\dots, x_{n-2}z_t, x_{n-1}x_{n-3}y_1,\dots ,x_{n-1}x_{n-3}y_s. \]
	  	Therefore $w_1=x_{n-2}z_1,\dots ,w_t=x_{n-2}z_t$ and $\alpha=t$ as desired.
\end{proof}

	\begin{Lemma}\label{lem:product of all As}
		Let $n\geq 4$. Then $x_{n-1}^2U\in G(J(P_n)^2)$ if and only if $U\in G(J(P_{n-2})^2)$.	  
\end{Lemma}
\begin{proof}
	($\Rightarrow$) Suppose that $x_{n-1}^2U\in G(J(P_n)^2)$. Then $x_{n-1}^2U=(x_{n-1}u_1)(x_{n-1}u_2)$ for some $u_1,u_2\in G(J(P_{n-2}))$ by Lemma~\ref{lem: rooted order for first power}. Let $V=v_1v_2$ for some $v_1,v_2\in G(J(P_{n-2}))$ such that $V|U$. Then by Lemma~\ref{lem: order preserved for P_n-2} we have $W=(x_{n-1}v_1)(x_{n-1}v_2)\in F(J(P_n)^2)$. Since $W|x_{n-1}^2U$ and $x_{n-1}^2U$ is a minimal generator we get $W=x_{n-1}^2U$. Therefore $V=U$ and $U$ is a minimal generator of $J(P_{n-2})^2$.
	
	($\Leftarrow$) Let $U\in G(J(P_{n-2})^2)$. Then $U=u_1u_2$ for some $u_1,u_2\in G(J(P_{n-2}))$. Then by Lemma~\ref{lem: order preserved for P_n-2} we have $x_{n-1}^2U=(x_{n-1}u_1)(x_{n-1}u_2)\in F(J(P_n)^2)$. Let $V\in G(J(P_n)^2)$ such that $V|x_{n-1}^2U$. By Lemma~\ref{lem: rooted order for first power} one can write $V=(x_{n-1}v_1)(x_{n-1}v_2)$ for some $v_1,v_2\in G(J(P_{n-2}))$. Then we get $v_1v_2|u_1u_2$. Since $U$ is a minimal generator, $v_1v_2=u_1u_2$. Therefore $V=x_{n-1}^2U\in G(J(P_n)^2)$ as desired.
\end{proof}
\begin{Lemma}\label{lem:product of all Bs}
		Let $n\geq 5$. Then $x_n^2x_{n-2}^2U\in G(J(P_n)^2)$ if and only if $U\in G(J(P_{n-3})^2)$.	
\end{Lemma}
\begin{proof}
 One can mimic the arguments in the proof of the previous lemma by using Lemma~\ref{lem: order preserved for P_n-3} instead of Lemma~\ref{lem: order preserved for P_n-2}. 
\end{proof}


\begin{Lemma}\label{lem: divisor of same type}
	Let $n\geq 7$ and let $uv, u'v'\in F(J(P_n)^2)$ with $u>_\MR v$ and $u'>_\MR v'$. Suppose that $u'v'$ divides $uv$. In the notation of Figure~\ref{branching} the following statements hold.
	\begin{enumerate}
		\item If $u\in \mathcal{A}$ and $v\in \mathcal{C}$, then $u'\in \mathcal{A}$ and $v'\in \mathcal{C}$.
		\item If $u\in \mathcal{B}$ and $v\in \mathcal{C}$, then $u'\in \mathcal{B}$ and $v'\in \mathcal{C}$.
	\end{enumerate}
\end{Lemma}
\begin{proof}

(1) First note that if $u\in \mathcal{A}$ and $v\in \mathcal{C}$, then $u \in x_{n-1} \MR(P_{n-2})$ and $v \in x_n x_{n-2} \MR(P_{n-3})$. Since $u'v'|uv$ and $u'>_\MR v'$, it shows that $u' \in x_{n-1} \MR(P_{n-2})$ and $v' \in x_n x_{n-2} \MR(P_{n-3})$.

Now we show that  $u'\in \mathcal{A}$ and $v'\in \mathcal{C}$. Indeed, if $u' \in \MB$, then $x_{n-2}^2 |u'v'$ but  $x_{n-2}^2 \nmid uv$. Therefore,  $u'\in \mathcal{A}$. Furthermore,  if $v' \in \MD$ then $x_{n-3}^2 |u'v'$ becuase $u' \in \MA$. But, again  $x_{n-3}^2 \nmid uv$. Hence $v'\in \mathcal{C}$, as required.	

(2) First note that if $u\in \mathcal{B}$ and $v\in \mathcal{C}$, then $u \in x_{n-1} \MR(P_{n-2})$ and $v \in x_n x_{n-2} \MR(P_{n-3})$. Since $u'v'|uv$ and $u'>_\MR v'$, it follows that $u' \in x_{n-1} \MR(P_{n-2})$ and $v' \in x_n x_{n-2} \MR(P_{n-3})$. Then, to show that $u'\in \mathcal{B}$ and $v'\in \mathcal{C}$, note that if $u' \in \MA$ or $v' \in \MD$, then $x_{n-3} |u'v'$ but  $x_{n-3}\nmid uv$.	
\end{proof}

\begin{Remark}\label{rk:max expression preserved in AC}
 Let $\MR(P_{n-2})=u_1,\dots ,u_m$. Furthermore, let $u_i, u_j$ be monomials such that $x_{n-1}u_i\in \mathcal{A}$ and $x_n u_j\in \mathcal{C}$ where $\mathcal{A}$, $\mathcal{C}$ are as in Figure~\ref{branching}. Let $u_iu_j$ be a maximal expression in $F(J(P_{n-2})^2)$ with $i\leq j$. Then $(x_{n-1}u_i)(x_nu_j)$ is maximal expression in $F(J(P_{n})^2)$. Indeed, otherwise from Lemma~\ref{lem: divisor of same type}, we see that the maximal expression of $x_{n-1}x_nu_iu_j$ is of the form $(x_{n-1}u_p)(x_nu_q)$ with $x_{n-1}u_p \in \MA$ and $x_nu_q \in \MC$.  From $x_{n-1}x_nu_iu_j= x_{n-1}x_nu_pu_q$, we see that $u_iu_j=u_pu_q$. Also, we have $x_{n-1}u_p >_{\MR} x_{n-1}u_i$ or $x_{n-1}u_p = x_{n-1}u_i$  and $x_nu_q >_{\MR} x_nu_j$ in $\MR(P_n)$. This shows that $u_p >_{\MR} u_i$ or $u_p =u_i$  and $u_q >_{\MR} u_j$ in $\MR(P_{n-2})$. This gives a contradiction to the fact that $u_iu_j$ is the maximal expression in $F(J(P_{n-2})^2)$.
\end{Remark}

\begin{Lemma}\label{lem: type AC minimal generator preserved}
	Let $u$ and $v$ be monomials such that $x_{n-1}u\in \mathcal{A}$ and $x_nv\in \mathcal{C}$ where $\mathcal{A}$, $\mathcal{C}$ are as in Figure~\ref{branching}. If $uv\in G(J(P_{n-2})^2)$, then $x_{n-1}x_nuv\in G(J(P_n)^2)$.
\end{Lemma}
\begin{proof}
Assume for a contradiction $uv\in G(J(P_{n-2})^2)$ but $x_{n-1}x_nuv\notin G(J(P_n)^2)$. Then there exists some $U \in G(J(P_n)^2)$, such that $U$ strictly divides $x_{n-1}x_nuv$. Then from Lemma~\ref{lem: divisor of same type}, we see that $U=(x_{n-1}u')( x_nv')$ for some $x_{n-1}u' \in \MA$ and $x_nv' \in \MC$. Then $u'v' \in F(J(P_{n-2})^2)$ and $u'v'$ strictly divides $uv$, which contradicts the hypothesis that $uv\in G(J(P_{n-2})^2)$. 
\end{proof}

\begin{Lemma}\label{lem: type AC order preserved}
	Let $u,u',v,v'$ be monomials such that $x_{n-1}u, x_{n-1}u'\in \mathcal{A}$ and $x_nv,x_nv'\in \mathcal{C}$ where $\mathcal{A}$, $\mathcal{C}$ are as in Figure~\ref{branching}. If $uv >_\MR u'v'$ in $F(J(P_{n-2})^2)$, then $x_{n-1}x_nuv >_\MR x_{n-1}x_nu'v'$ in $F(J(P_{n})^2)$. 
\end{Lemma}

\begin{proof}
  Branching of $\MR(P_n)$ in Figure~\ref{branching}, shows that if  $x_{n-1}u, x_{n-1}u'\in \mathcal{A}$ and $x_nv,x_nv'\in \mathcal{C}$, then we have $u, u' \in x_{n-3} \MR(P_{n-4})$ and $v, v' \in  x_{n-2}x_{n-4} \MR(P_{n-5})$. Then by the definition of $>_{\MR}$, it follows that $u>_{\MR } v$ and $u'>_{\MR} v'$ in $ \MR(P_{n-2})$.
 
 Note that because of Lemma~\ref{lem: divisor of same type} (1) we may assume that $uv$ and $u'v'$ are maximal expressions in $F(J(P_{n-2})^2)$. Then Remark~\ref{rk:max expression preserved in AC} implies that 
 the expressions $(x_{n-1}u)(x_nv)$ and $(x_{n-1}u')(x_nv')$ are maximal in $F(J(P_n)^2)$.

Let $uv >_\MR u'v'$ in $F(J(P_{n-2})^2)$, then by definition of $>_{\MR}$, we have either $u>_{\MR} u'$ or $u= u'$ and $v>_{\MR} v'$ in  $\MR(P_{n-2})$. If  $u>_{\MR} u'$, then by Remark~\ref{rk:max expression preserved P_{n-2}} $x_{n-1}u>_\MR x_{n-1}u'$ in $\MR(P_n)$. If  $v>_{\MR} v'$ in  $\MR(P_{n-2})$, then $x_{n-1}v$ appears before $x_{n-1}v'$ in the sublist $\mathcal{B}$. This implies $x_nv$ appears before $x_nv'$ in the sublist $\mathcal{C}$ and thus $x_nv>_\MR x_nv'$ as desired.
\end{proof}

\begin{Lemma}\label{lem: type BC minimal generator preserved}
	Let $u,v\in G(J(P_{n-5}))$. If $uv\in G(J(P_{n-5})^2)$, then $$(x_{n-1}x_{n-2}x_{n-4}u)(x_nx_{n-2}x_{n-4}v)\in G(J(P_n)^2).$$
\end{Lemma}
\begin{proof}
A similar argument as in Lemma~\ref{lem: type AC minimal generator preserved} gives the desired result.
\end{proof}

\begin{Remark}\label{rk:max expression preserved in BC}
Let $\MR(P_{n-5})=u_1>_{\MR}\dots >_{\MR} u_p$. If  $u_iu_j$ is a maximal expression with $i\leq j$ in $F(J(P_{n-5})^2)$ then together with Lemma~\ref{lem: divisor of same type} and a similar explanation as in Remark~\ref{rk:max expression preserved in AC} we see that $(x_{n-1}x_{n-2}x_{n-4}u_i)(x_nx_{n-2}x_{n-4}u_j)$ is a maximal expression in $F(J(P_{n})^2)$.
\end{Remark}

\begin{Lemma}\label{lem: type BC order preserved}
	Let $\MR(P_{n-5})=u_1>_{\MR}\dots >_{\MR} u_p$.	 If $u_iu_k >_\MR u_ju_l$ in $F(J(P_{n-5})^2)$ for some $i,j,k$ and $l$, then $x_nx_{n-1}x_{n-2}^2x_{n-4}^2u_iu_k >_\MR x_nx_{n-1}x_{n-2}^2x_{n-4}^2u_ju_l$ in $F(J(P_{n})^2)$. 
\end{Lemma}
\begin{proof}
We may assume that $u_iu_k$ and $u_ju_l$ are maximal expressions with $i \leq k$ and $j\leq l$. By Remark~\ref{rk:max expression preserved in BC} the expressions
\[(x_{n-1}x_{n-2}x_{n-4}u_i)(x_nx_{n-2}x_{n-4}u_k) \text{ and }  (x_{n-1}x_{n-2}x_{n-4}u_j)(x_nx_{n-2}x_{n-4}u_l)\]
are both maximal in $F(J(P_{n})^2)$. Given that $u_iu_k >_\MR u_ju_l$ in $F(J(P_{n-5})^2)$, we have that either $i <j$ or $i=j$ and $k < l$.

If $i<j$, then $x_{n-1}x_{n-2}x_{n-4}u_i$ appears before $x_{n-1}x_{n-2}x_{n-4}u_j$ in the sublist $\mathcal{B}$. Therefore we get $x_{n-1}x_{n-2}x_{n-4}u_i>_\MR x_{n-1}x_{n-2}x_{n-4}u_j$ in $\MR(P_n)$.

If $k<l$, then $x_nx_{n-2}x_{n-4}u_k$ appears before $x_nx_{n-2}x_{n-4}u_l$ in the sublist $\mathcal{C}$. Therefore  $x_nx_{n-2}x_{n-4}u_k >_\MR x_nx_{n-2}x_{n-4}u_l$ in $\MR(P_n)$. Thus the result follows by the definition of rooted order.

\end{proof}


\section{$2$-fold products of $J(P_n)$ versus minimal generators of $J(P_n)^2$}\label{sec:2fold products and minimal generators}

To be able to prove our main result, we need to filter out those $2$-fold products which are not in  $G(J(P_n)^2)$. The next lemma gives a sufficient condition for a $2$-fold product to be a non-minimal generator. While reading its proof, we advise the reader to keep in mind that a minimal vertex cover of a path cannot contain $3$ consecutive vertices.

\begin{Lemma}\label{lem:sufficient condition for nonminimal gens}
	Let $n\geq 5$. Let $u$ and $v$ be minimal generators of $J=J(P_n)$ such that $x_{n-1}x_{n-4}\mid u$ and $x_nx_{n-3}\mid v$. Then $uv$ is not a minimal generator of $J^2$. Moreover, there exists a $2$-fold product $pw\in F(J^2)$ such that $pw\mid uv$ and $pw>_{\MR} uv$.
\end{Lemma}
\begin{proof}
	 Let $X=x_nx_{n-1}x_{n-2}x_{n-3}x_{n-4}$. Since $v$ is a minimal generator $x_{n-1}\nmid v$. This implies $x_{n-2}|v$. Then because $v$ is minimal, we get $x_{n-4}\nmid v$. Thus \[\gcd(v,X)=x_nx_{n-2}x_{n-3}.\]
	By minimality of $u$ we get $x_n\nmid u$. Since $x_{n-4}$ divides $u$ we see that $u$ is divisible by either $x_{n-2}$ or $x_{n-3}$, but not both. Therefore
	\[\gcd(u,X)=x_{n-1}x_{n-2}x_{n-4} \ \text{ or } \ \gcd(u,X)=x_{n-1}x_{n-3}x_{n-4}.\]
	
	Let $u'v'$ be a maximal expression for $uv$ for some $u'>_{\MR} v'$. We claim that $\gcd(u,X)=\gcd(u',X)$ and $\gcd(v,X)=\gcd(v',X)$. First observe that since $u'>_{\MR} v'$ the variable $x_{n-1}$ divides $u'$ but not $v'$. This implies $x_n$ divides $v'$ but not $u'$ because both $u'$ and $v'$ are minimal generators. We consider cases:
	
	\emph{Case 1}: Suppose $\gcd(u,X)=x_{n-1}x_{n-2}x_{n-4}$.  Since $uv=u'v'$ we see that $x_{n-2}$ divides both $u'$ and $v'$. Now by minimality of $u'$ we must have $x_{n-3}\nmid u'$. This implies $x_{n-3}\mid v'$. Minimality of $v'$ requires $x_{n-4}\nmid v'$. Then $x_{n-4}\mid u'$. Hence $\gcd(v',X)=x_nx_{n-2}x_{n-3}$ and $\gcd(u',X)=x_{n-1}x_{n-2}x_{n-4}$ as desired.
	
	\emph{Case 2:}  Suppose $\gcd(u,X)=x_{n-1}x_{n-3}x_{n-4}$. Then since $x_{n-3}^2\mid uv=u'v'$ we see that  $x_{n-3}$ divides both $u'$ and $v'$. By minimality of $u'$ observe that $x_{n-2}\nmid u'$. Since $x_{n-2} \mid uv=u'v'$ we get $x_{n-2}\mid v'$. Now by minimality of $v'$ we get $x_{n-4}\nmid v'$ which implies $x_{n-4}\mid u'$. Hence $\gcd(v',X)=x_nx_{n-2}x_{n-3}$ and $\gcd(u',X)=x_{n-1}x_{n-3}x_{n-4}$ which completes the proof of our claim.

	Observe that $w=(v'x_{n-1})/(x_nx_{n-2})$ is a minimal vertex cover of $P_n$. Again we consider cases:
	
	\emph{Case (a)}: Suppose $x_{n-3}\mid u'$. Observe that $p=(u'x_{n-2}x_n)/(x_{n-1}x_{n-3})$ is a minimal vertex cover of $P_n$ and $pwx_{n-3}=u'v'$. 
	
	\emph{Case (b)}: Suppose $x_{n-2}\mid u'$. Observe that $p=(u'x_n)/x_{n-1}$ is a minimal vertex cover of $P_n$ and $pwx_{n-2}=u'v'$.

Observe that in each case $w>_{\MR} p$. Since both $w$ and $u'$ are in $x_{n-1}\MR(P_{n-2})$, applying Lemma~\ref{lem:properties of rooted list} (2) to $P_{n-2}$ we see that $w>_{\MR}u'$. Therefore $pw>_{\MR} u'v'$.
\end{proof}


We will need the next result to detect some of $2$-fold products which yield non-minimal generators or non-maximal expressions.

\begin{Proposition}\label{prop:condition for non-minimal gen or non-maximal expression}
    Let $\MR(P_{n})=u_1>_{\MR}\dots >_{\MR} u_k$ where $n\geq 2$.
	Let $1< i<j\leq k$. Suppose $u_j$ contains a variable from $(u_1,\dots ,u_{i-1}):(u_i)$. Then either $u_iu_j$ is not a minimal generator of $J(P_n)^2$ or $u_iu_j$ is not a maximal $2$-fold expression.
\end{Proposition}
\begin{proof}
	We proceed by induction on $n$. The statement holds for $n\leq 7$, see Table~\ref{Table2} for verification. Therefore let us assume that $n\geq 8$. 
	
	Observe that if $u_iu_j$ is divisible by $x_{n-1}^2$  then the result follows from Lemma~\ref{lem: order preserved for P_n-2}\,, Lemma~\ref{lem:product of all As}\,  and the induction assumption on $P_{n-2}$. If $u_iu_j$ is divisible by $x_n^2$  then the result follows from Lemmas~\ref{lem: order preserved for P_n-3}\,, \ref{lem: colon contains the second variable from the end}\,, \ref{lem:product of all Bs} and the induction assumption on $P_{n-3}$. Therefore we may assume that $x_{n-1}\mid u_i$ and $x_n\mid u_j$. Now consider the following rooted lists:
	\begin{align*}
	\MR(P_{n-4}): & \ v_1 >_{\MR}\dots >_{\MR}v_{\ell}\\
	\MR(P_{n-5}): & \ w_1>_{\MR}  \dots >_{\MR}w_m\\
	\MR(P_{n-6}): & \ p_1>_{\MR} \dots >_{\MR}p_q.
	\end{align*}
	Note that $\MR(P_n)$ is the join of the list $\mathcal{A}, \mathcal{B}, \mathcal{C},\mathcal{D}$ in Figure~\ref{branching}\, in the given order. We consider the following cases.
	
	\emph{Case 1:} Suppose that $u_j\in \mathcal{D}$. If $x_{n-4}\mid u_i$, then the result follows from Lemma~\ref{lem:sufficient condition for nonminimal gens}. So, let us assume that $x_{n-4}\nmid u_i$ as well. Observe now that $u_i\in \mathcal{A}$. Note that by Lemma~\ref{lem:properties of rooted list}~(3) we have
	$$\MR(P_{n-4})=x_{n-5}p_1,\dots ,x_{n-5}p_q, x_{n-4}w_1,\dots ,x_{n-4}w_{\alpha}$$
	for some $\alpha\leq m$. Therefore $u_i=x_{n-1}x_{n-3}x_{n-5}p_i$ and  $u_j=x_{n}x_{n-2}x_{n-3}x_{n-5}p_{\beta}$
	for some $\beta \leq q$. Clearly we have
	$$(u_1,\dots ,u_{i-1}):(u_i)=(p_1,\dots ,p_{i-1}):(p_i)$$ and $\beta\neq i$. Thus $p_{\beta}$ contains a variable generator of $(p_1,\dots ,p_{i-1}):(p_i)$. 
	
	\emph{Case 1.1:} If $\beta <i$, then observe that we can produce a new expression 
	$$\displaystyle u_iu_j= \underbrace{(x_{n-1}x_{n-3}x_{n-5}p_\beta)}_{u_\beta}\underbrace{(x_{n}x_{n-2}x_{n-3}x_{n-5}p_i)}_{u_\gamma} $$
	for some $\gamma$. Then $u_\beta >_\MR u_\gamma$ and $u_\beta >_\MR u_i$ which implies that the expression $u_iu_j$ is not maximal.  
	
	\emph{Case 1.2:} Let $\beta > i$. Then by induction assumption either $p_ip_{\beta}$ is not minimal generator of $J(P_{n-6})^2$ or the expression $p_ip_{\beta}$ is not maximal. Any minimal generator of $J(P_{n-6})^2$ which divides $p_ip_{\beta}$ or any $2$-fold expression which is greater than $p_ip_{\beta}$ can be multiplied by the appropriate variables to obtain the desired conclusion for $u_iu_j$. 
	
	\emph{Case 2:} Suppose $u_j\in\mathcal{C}$ so that $u_j=x_nx_{n-2}x_{n-4}w_s$ for some $s\geq 1$.
	
	\emph{Case 2.1:} Suppose $u_i\in\mathcal{B}$ so that $u_i=x_{n-1}x_{n-2}x_{n-4}w_t$ for some $t$. Observe that since $u_i\in\mathcal{B}$ we have $i\geq \ell+1$. We claim that $i> \ell+1$. Assume for a contradiction $i=\ell+1$. Then by Lemma~\ref{lem: colon contains the second variable from the end}\, we get
	$$(u_1,\dots ,u_{i-1}):(u_i)=(x_{n-3})$$
	which implies that $x_{n-3}\mid u_j$, a contradiction. Hence $i> \ell+1$ indeed. Applying Lemma~\ref{lem: colon contains the second variable from the end} we obtain
	$$(u_1,\dots ,u_{i-1}):(u_i)=(u_{\ell+1},\dots ,u_{i-1}):(u_i) +(x_{n-3}).   $$
	
	Now since $(u_{\ell+1},\dots ,u_{i-1}):(u_i)=(w_1,\dots, w_{t-1}):(w_t)$, there exists a variable generator of this ideal dividing $u_j$ and thus dividing $w_s$. Clearly $s\neq t$. If $s>t$, then by induction assumption on $P_{n-5}$, either $w_tw_s$ is a non-minimal generator or is not a maximal expression and the result follows as in Case $1$. Lastly, suppose that $s<t$. Then we obtain a different expression for $u_iu_j$ as follows:
		\begin{equation*}
	\begin{split}
	u_iu_j& = (x_{n-1}x_{n-2}x_{n-4}w_t)(x_nx_{n-2}x_{n-4}w_s)  \\
	& = (x_{n-1}x_{n-2}x_{n-4}w_s)(x_nx_{n-2}x_{n-4}w_t)  \\
	& = u_{s+|\mathcal{A}|}u_{t+|\mathcal{A}|+|\mathcal{B}}| \qquad \text{by Figure~\ref{branching}}\\
	& = u_{s+\ell}u_{t+\ell+m}.
	\end{split}
	\end{equation*} 
	
	Since $i=t+\ell$ we have $u_{s+\ell}>_\MR u_i$ and the expression $u_iu_j$ is not maximal.
	
	\emph{Case 2.2:} Suppose $u_i\in \mathcal{A}$ so that $u_i=x_{n-1}x_{n-3}v_i$. Observe that by Lemma~\ref{lem:properties of rooted list}~(1) the variable $x_{n-4}$ is not a generator of the ideal $$(u_1,\dots ,u_{i-1}):(u_i)=(v_1,\dots ,v_{i-1}):(v_i)$$ and $w_s$ is divisible by a variable generator of $(v_1,\dots ,v_{i-1}):(v_i)$.
	
	By induction assumption $(x_{n-3}v_i)(x_{n-2}x_{n-4}w_s)$ is either a non-minimal generator of $J(P_{n-2})^2$ or a non-maximal expression. If it is not a minimal generator, then it is divisible by some $(x_{n-3}v_{\alpha})(x_{n-2}x_{n-4}w_{\beta})$ and $v_{\alpha}w_{\beta}\mid v_iw_s$. In such case, multiplying $v_{\alpha}w_{\beta}$ by the appropriate variables one can see that $u_iu_j$ is not a minimal generator. Lastly, observe that if  $(x_{n-3}v_i)(x_{n-2}x_{n-4}w_s)$ is a non-maximal expression, then so is $u_iu_j$.
	\end{proof}

\begin{Lemma}\label{lem:pure powers stay mingen}
	Let $I$ be a squarefree monomial ideal and let $u$ be a minimal generator of $I$. Then $u^s$ is a minimal generator of $I^s$ for all $s$.
\end{Lemma}

\begin{proof}
	Suppose $v=v_1^{a_1}\dots v_q^{a_q}\in G(I^s)$ where $v_1,\dots , v_q$ are some minimal generators of $I$ and $a_1+\dots +a_q=s$ and $a_i>0$ for all $i$. Suppose $v$ divides $u^s$. Then each $v_i$ divides $u$ since $v_i$ is squarefree. Then by minimality of $u$ we get $u=v_i$ for all $i=1,\dots ,q$.
\end{proof}

\begin{Remark}
Note that in the lemma above the squarefreeness assumption cannot be omitted. For example, if $I=(a^2bc, b^2, c^2)$ then $(a^2bc)^2\notin G(I^2)$.
\end{Remark}

The following lemma is of crucial importance to prove the main result stated in Theorem~\ref{main thm}.
 \begin{Lemma}\label{weak}
 Let $U \in F(J(P_n)^2)\setminus G(J(P_n)^2)$. Then there exists  $V\in G(J(P_n)^2)$ such that $V>_{\MR}  U$ and $V|U$.
\end{Lemma}

\begin{proof}
We will prove the assertion by applying induction on $n$. The statement holds for $n\leq 7$, see Table~\ref{Table2} for verification. Assume that $n \geq 8$.

Let $\MR(P_{n})=u_1>_{\MR}\dots >_{\MR} u_f$. Because of Lemma~\ref{lem:pure powers stay mingen}, we may assume that $U=u_iu_j$ is a maximal expression for some $i<j$.
From the Figure~\ref{branching}, which describes the branching of rooted order of minimal generators of $J(P_n)$, we see that we have the following three possibilities.  

\begin{itemize}
\item[(1)] $u_i, u_j \in x_{n-1}\MR(P_{n-2})$
\item[(2)] $u_i, u_j \in x_{n}x_{n-2}\MR(P_{n-3})$
\item[(3)] $u_i \in x_{n-1}\MR(P_{n-2})$ and  $u_j \in x_nx_{n-2}\MR(P_{n-3})$.
\end{itemize}

Since $U \notin G(J(P_n)^2)$, there exists $U' \in G(J(P_n)^2)$  such that $U'$ strictly divides $U$. Let $U' =u_pu_q$ be a maximal expression for some $p\leq q$.   Now we discuss each of the above possibilities separately. Let $\MR(P_{n-2})=v_1>_{\MR}\dots >_{\MR} v_d$ and $\MR(P_{n-3})=l_1>_{\MR}\dots >_{\MR} l_e$.

(1): Let $u_i, u_j \in x_{n-1}\MR(P_{n-2})$. Then, $u_p, u_q \in x_{n-1}\MR(P_{n-2})$ because $U'|U$. Also, in this case we have $U=(x_{n-1}v_{i'})(x_{n-1}v_{j'})$ and $U'=(x_{n-1}v_{p'})(x_{n-1}v_{q'})$ for some $v_{i'},v_{j'},v_{p'},v_{q'} \in  \MR(P_{n-2})$. Then the monomial $v_{p'}v_{q'}$ strictly divides $v_{i'}v_{j'}$. By induction hypothesis on $P_{n-2}$, there exists $v_{r'}v_{s'}\in  G(J(P_{n-2})^2)$ such that $v_{r'}v_{s'}| v_{i'}v_{j'}$ and $v_{r'}v_{s'}>_{\MR}  v_{i'}v_{j'}$. Let $V=(x_{n-1}v_{r'})(x_{n-1}v_{s'})$. By Lemma~\ref{lem:product of all As}, we see that $V \in G(J(P_{n})^2)$. Note that  $V|U$ and by following Lemma~\ref{lem: order preserved for P_n-2} we get $V >_{\MR}  U$, as required. 
 
(2): Let $u_i, u_j \in x_{n}x_{n-2}\MR(P_{n-3})$. Then, $u_p, u_q \in x_{n}x_{n-2}\MR(P_{n-3})$ because $U'|U$. Also, in this case we have $U=(x_nx_{n-2}l_{i'})(x_nx_{n-2}l_{j'})$ and $U'=(x_nx_{n-2}l_{p'})(x_nx_{n-2}l_{q'})$ for some $l_{i'},l_{j'},l_{p'},l_{q'} \in  \MR(P_{n-3})$. Then the monomial $l_{p'}l_{q'}$ strictly divides $l_{i'}l_{j'}$. By induction hypothesis on $P_{n-3}$, there exists $l_{r'}l_{s'}\in  G({J(P_{n-3})}^2)$ such that $l_{r'}l_{s'}| l_{i'}l_{j'}$ and $l_{r'}l_{s'}>_{\MR} l_{i'}l_{j'}$. Let $V=(x_nx_{n-2}l_{r'})(x_nx_{n-2}l_{s'})$. Also, by Lemma~\ref{lem:product of all Bs}, we see that $V \in G(J(P_{n})^2)$. Note that $V|U$ and by following Lemma~\ref{lem: order preserved for P_n-3} we get $V >_{\MR}  U$, as required.

(3): If $u_i \in x_{n-1}\MR(P_{n-2})$ and  $u_j \in x_nx_{n-2}\MR(P_{n-3})$ then again from Figure~\ref{branching}, we see that either $u_i \in \MA$ or $u_i \in \MB$, and either $u_j \in \MC$ or $u_j \in \MD$. We list these four cases in the following way.  
 
 \begin{itemize}
 \item[(a)] $u_i \in \MA$ and $u_j \in \MC$ ;
 \item[(b)] $u_i \in \MA$ and $u_j \in \MD$;
 \item[(c)] $u_i \in \MB$ and $u_j \in \MC$;
 \item[(d)] $u_i \in \MB$ and $u_j \in \MD$.
 \end{itemize}
 We define the following rooted lists:
 \begin{align*}
 \MR(P_{n-4})= & \ a_1>_{\MR}\dots >_{\MR} a_g \\
 \MR(P_{n-5})= &\ b_1>_{\MR}\dots >_{\MR} b_h \\
 \MR(P_{n-6})= &\ c_1>_{\MR}\dots >_{\MR} c_k.
 \end{align*}
\emph{Case (a):} If $u_i \in \MA$ and $u_j \in \MC$, then $U=(x_{n-1}x_{n-3}a_{i'})(x_nx_{n-2}x_{n-4} b_{j'} )$ for some $a_{i'} \in \MR(P_{n-4})$ and $b_{j'} \in \MR(P_{n-5})$. Since $U'|U$, by Lemma~\ref{lem: divisor of same type} we get $u_p \in \MA$ and $u_q \in \MC$. Then $U'=(x_{n-1}x_{n-3}a_{p'})(x_nx_{n-2}x_{n-4} b_{q'} )$ for some $a_{p'} \in \MR(P_{n-4})$ and $b_{q'} \in \MR(P_{n-5})$. Moreover, $U'|U$ gives $a_{p'} b_{q'} | a_{i'} b_{j'}$.
 
 Note that $x_{n-3}a_{i'}, x_{n-3}a_{p'},  x_{n-2}x_{n-4}b_{j'},x_{n-2}x_{n-4}b_{q'} \in \MR(P_{n-2})$ and the monomial 
  $(x_{n-3}a_{p'})(x_{n-2}x_{n-4}b_{q'})$ strictly divides $Y=(x_{n-3}a_{i'})(x_{n-2}x_{n-4}b_{j'}) $ which shows that $ Y \in F(J(P_{n-2})^2)\setminus G( J(P_{n-2})^2)$. Then by induction hypothesis on $P_{n-2}$, we know that there exists  $Y'\in  G(J(P_{n-2})^2)$ such that $Y'| Y$ and $Y' >_{\MR}  Y$.
  
   Observe that $Y'=(x_{n-3}a_{i''})(x_{n-2}x_{n-4}b_{j''})$ for some $a_{i''}\in \MR(P_{n-4})$ and $b_{j''}\in \MR(P_{n-5})$. Let $V=x_nx_{n-1}Y'$ and note that from Lemma~\ref{lem: type AC minimal generator preserved}, we have $V \in G(J(P_n)^2)$. Clearly $V$ divides $U$. From Lemma~\ref{lem: type AC order preserved} it follows that $V >_{\MR}  U$ as desired.

 \emph{Case (b):} If $u_i \in \MA$ and $u_j \in \MD$, then $U=(x_{n-1}x_{n-3}a_{i'} )(x_n x_{n-2}x_{n-3}x_{n-5} c_{j'})$ for some $a_{i'} \in \MR(P_{n-4})$ and $c_{j'} \in \MR (P_{n-6})$. Since $U'|U$, and $U'=u_p u_q$ one can see that $u_p \in  \MA$. Also, $u_q \in \MC$ or $u_q \in \MD$.
 
 If  $u_p \in  \MA$ and $u_q \in \MC$, then $U'= (x_{n-1}x_{n-3} a_{p'})(  x_nx_{n-2}x_{n-4} b_{q'})$  for some $a_{p'} \in \MR(P_{n-4})$ and $b_{q'} \in \MR (P_{n-5})$. Then, $U'|U$ gives
 	\[
 	(x_{n-1}x_{n-3} a_{p'})(  x_nx_{n-2}x_{n-4} b_{q'}) | (x_{n-1}x_{n-3}a_{i'} )(x_n x_{n-2}x_{n-3}x_{n-5} c_{j'})
 	\]
 	which implies
 	\[
 	x_{n-4} a_{p'} b_{q'}| x_{n-3}x_{n-5} a_{i'} c_{j'}.
 	\]
 	Therefore $x_{n-4} | a_{i'}$ because $c_{j'} \in \MR (P_{n-6})$.  It shows that $x_{n-1}x_{n-4}\mid u_i$. Then by Lemma~\ref{lem:sufficient condition for nonminimal gens}, we get the desired result.

 If  $u_p \in \MA$ and $u_q \in \MD$, then $U'=(x_{n-1}x_{n-3}a_{p'} )(x_n x_{n-2}x_{n-3}x_{n-5} c_{q'})$ for some $a_{p'} \in \MR(P_{n-4})$ and $c_{q'} \in \MR(P_{n-6})$.
 \begin{figure}[htbp]
 	\includegraphics[width = 15cm]{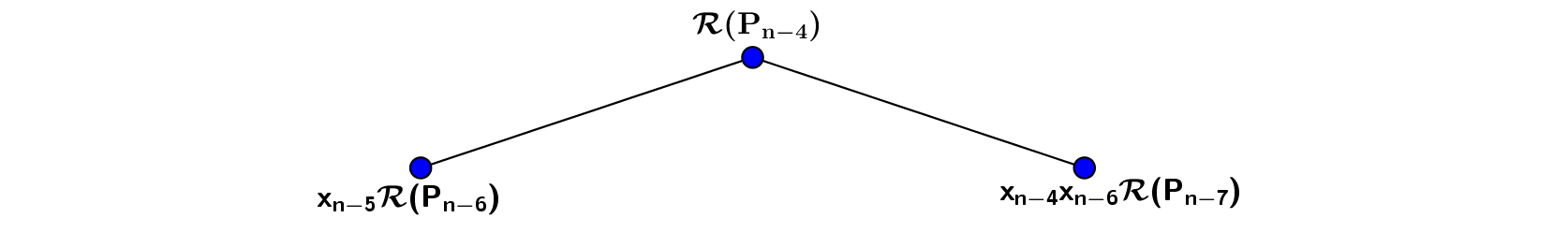}
 	\caption{Branching of $\MR(P_{n-4})$}
 	\label{branching2}
 \end{figure}
 Keeping Figure~\ref{branching2} in mind, one can check that either $x_{n-4}$ divides $a_{i'}$ or $a_{i'} \in x_{n-5} \MR(P_{n-6})$. If $x_{n-4}$ divides $a_{i'}$ then $u_i$ is divisible by $x_{n-1}x_{n-4}$ and the result follows from Lemma~\ref{lem:sufficient condition for nonminimal gens}.
 
 Therefore let us assume that $a_{i'} \in x_{n-5} \MR(P_{n-6})$. Then it is not hard to show that $a_{p'} \in x_{n-5} \MR(P_{n-6})$ as well. Then a similar argument as in Case~(c) shows that we can find $V \in G(J(P_{n})^2)$ such that  $V|U$ and $V>_{\MR} U$.

\emph{Case (c):}
If $u_i \in \MB$ and $u_j \in \MC$, then $u_i=x_{n-1} x_{n-2}x_{n-4}b_{i'}$ and $u_j=x_n x_{n-2}x_{n-4} b_{j'}$ for some $b_{i'}, b_{j'}  \in \MR(P_{n-5})$. Since $U'|U$, it follows from Lemma~\ref{lem: divisor of same type}~(2) that $U'=(x_{n-1} x_{n-2}x_{n-4}b_{p'} )(x_n x_{n-2}x_{n-4} b_{q'})$ for some $b_{p'}, b_{q'}  \in \MR(P_{n-5})$. Moreover, $b_{p'} b_{q'}$ strictly divides $b_{i'} b_{j'}$. This shows that $b_{i'} b_{j'} \in F(J(P_{n-5})^2)\setminus  G(J(P_{n-5})^2)$.   Then by induction hypothesis, we know that there exists  $Y\in G( J(P_{n-5})^2)$ such that $Y| b_{i'} b_{j'} $ and $Y>_{\MR} b_{i'} b_{j'}$. Then by Lemma~\ref{lem: type BC order preserved} we get 
\[x_nx_{n-1}x_{n-2}^2x_{n-4}^2Y >_\MR x_nx_{n-1}x_{n-2}^2x_{n-4}^2b_{i'}b_{j'} \ \text{ in } F(J(P_{n})^2).\]
Now, setting $V=x_nx_{n-1}x_{n-2}^2x_{n-4}^2Y $, the line above becomes
\[V>_\MR U \text{ in } F(J(P_{n})^2).\]
Clearly $V|U$. Because of Lemma~\ref{lem: type BC minimal generator preserved} we get $V \in G(J(P_n)^2)$ as desired. 

\emph{Case (d):} If $u_i \in \MB$ and $u_j \in \MD$, then  by Lemma ~\ref{lem:sufficient condition for nonminimal gens}, we get the desired result. 
\end{proof}
\section{Linear quotients of second power of $J(P_n)$}\label{sec:main results}
We are now ready to prove our main theorem.

\begin{Theorem}\label{main thm}
Let $G(J(P_n)^2)=\{U_1, \ldots, U_p\}$. Then $J(P_n)^2$ has linear quotients with respect to $U_1 >_{\MR} \ldots >_{\MR} U_p$.
\end{Theorem}
\begin{proof} We will prove the assertion by applying induction on $n$. The statement holds for $n\leq 5$, see Table~\ref{Table2} for verification. Suppose that $n\geq 5$. We need to show that  $(U_1,\dots ,U_{r-1}):(U_r)$ is generated by variables, for all $2 \leq r \leq p$. Let $\MR(P_{n-2})=m_1>_{\MR}\dots >_{\MR} m_a$ and $\MR(P_{n-3})=l_1>_{\MR}\dots >_{\MR} l_b$.
	

\emph{Case 1:} Suppose that $x_n^2$ divides $U_r$. Let us assume that $U_r$ has the maximal expression $U_r=(x_n x_{n-2} l_i)(x_n x_{n-2} l_j)$ for some $l_i, l_j \in  \MR(P_{n-3})$ with $i\leq j$. First, we claim that $x_{n-1}$ is a generator of $(U_1,\dots ,U_{r-1}):(U_r)$. 

 By Lemma~\ref{lem:every b contains some a} there exits $ m_q \in \MR(P_{n-2})$  such that $m_q|x_{n-2} l_i$.  Let $$V=(x_{n-1}m_q)(x_nx_{n-2}l_j).$$ Observe that $V:U_r=x_{n-1}$. Moreover, if $V \in G(J(P_n)^2)$, then $V >_{\MR} U_r$. Otherwise, by Lemma~\ref{weak}, there exist $U_k$ with $1 \leq k \leq r-1$ such that $U_k |V$ and $U_k >_{\MR} V$. Then $U_k : U_r= x_{n-1}$ which proves the claim. 

Let $\MR(J(P_{n-3})^2)=L_1>_{\MR} L_2>_{\MR}\dots >_{\MR}L_s$ and $L_t=l_il_j$. 
Now we will show that
\[
(U_1,\dots ,U_{r-1}):(U_r) =(x_{n-1})+(L_1,L_2,\dots, L_{t-1}):(L_t).
\]
Observe that the proof will be complete once we prove the equality above because of induction assumption on $P_{n-3}$. Combining Lemma~\ref{lem: order preserved for P_n-3}, Lemma~\ref{lem:product of all Bs} and the claim that has been proved, we obtain
\[
(x_{n-1})+(L_1,L_2,\dots, L_{t-1}):(L_t) \subseteq (U_1,\dots ,U_{r-1}):(U_r).
\]

It remains to show that the reverse inclusion holds. Note that for each $\ell \leq r-1$, the monomial $U_{\ell}$ is divisible by either $(x_{n-2}x_n)^2$ or $x_{n-1}$ because of definition of rooted order. If $x_{n-1} | U_{\ell}$, then it is easy to see that in this case $U_{\ell}:U_r \in (x_{n-1})$. If $(x_{n-2}x_n)^2 | U_{\ell}$, then by Lemma~\ref{lem:product of all Bs}, we have $U_{\ell}/(x_{n-2}x_n)^2=L_k$ for some $k$. Clearly we have $U_{\ell}:U_r=L_k:L_t$. Furthermore, because $U_{\ell} >_{\MR} U_{r}$, by Lemma~\ref{lem: order preserved for P_n-3} we get $L_k >_{\MR} L_t$ which completes the proof in this case.


\medskip
\emph{Case 2:} Suppose that $x_{n-1}^2$ divides $U_r$. Let $U_r=(x_{n-1}m_i)(x_{n-1}m_j)$ be the maximal expression for some $m_i, m_j \in \MR(P_{n-2})$ with $i\leq j$. Then the monomial $m_im_j$ is also in its maximal expression by Remark~\ref{rk:max expression preserved P_{n-2}}. Then Lemma~\ref{lem:product of all As} implies $m_im_j \in G(J(P_{n-2})^2)$. 
	Let $\MR(J(P_{n-2})^2)=M_1>_{\MR}\dots >_{\MR} M_s$. 
	Then $m_im_j=M_t$, for some $1 < t \leq s$. Note that $1<t$, because if $t=1$ then $r=1$ which is not true. By induction hypothesis, $(M_1,\dots , M_{t-1}):(M_t)$ is generated by variables. We claim that
\[
(M_1,\dots ,M_{t-1}):(M_t)=(U_1,\dots ,U_{r-1}):(U_r).
\]

By Remark~\ref{rk:max expression preserved P_{n-2}} and Lemma~\ref{lem:product of all As} it is clear that
\[
(M_1,\dots ,M_{t-1}):(M_t)  \subseteq (U_1,\dots ,U_{r-1}):(U_r) .
\]
	
We need to show the reverse inclusion. Observe that for every $\ell \leq r-1$, the monomial $U_{\ell}$ is divisible by either $x_{n-1}^2$ or $x_{n-1}x_n$ because of definition of rooted order. If $x_{n-1}^2$ divides $U_{\ell}$, then again by Lemma~\ref{lem:product of all As} we get $U_{\ell}/x^2_{n-1}=M_k$ for some $k$. Clearly, $U_{\ell}:U_r=M_k:M_t$. Therefore, it remains to show that $k<t$. Lemma~\ref{lem: order preserved for P_n-2} together with $U_{\ell} >_{\MR} U_{r}$ implies $M_k >_{\MR} M_t$ as desired.

If  $x_{n-1}x_n$ divides $U_{\ell}$, then we may assume that $U_{\ell}=(x_{n-1}m_h)(x_nx_{n-2}l_q)$ is the maximal expression for some $m_h \in \MR(P_{n-2})$ and $l_q \in \MR(P_{n-3})$. Then by Lemma~\ref{lem:every b contains some a}, there exists $m_v \in \MR(P_{n-2})$ such that $m_v | x_{n-2}l_q$.

Note that since $U_{\ell}>_{\MR}U_r$ we must have $m_h>_{\MR} m_i,m_j$ in $\MR(P_{n-2})$ by Lemma~\ref{lem: order preserved for P_n-2}. Now consider 
\[ P=(x_{n-1}m_h)(x_{n-1} m_v).
\]
	
If $P \in G(J(P_n)^2)$, then $P>_{\MR}U_r$ and $P:U_r\in(M_1,\dots ,M_{t-1}):(M_t)$. Since $P:U_r$ divides $U_{\ell}:U_r$ it follows that $U_{\ell}:U_r\in (M_1,\dots ,M_{t-1}):(M_t)$.
	
If $P \notin G(J(P_n)^2)$, then by Lemma~\ref{weak}, there exists  $U_\alpha \in G(J(P_n)^2)$ such that $U_\alpha|P$ and $U_\alpha>_{\MR} P$. Thus $U_\alpha>U_r$ and $U_\alpha:U_r\in(M_1,\dots ,M_{t-1}):(M_t)$. Since $U_\alpha:U_r$ divides $P:U_r$ and $P:U_r$ divides $U_{\ell}:U_r$, we have $U_\alpha:U_r$ divides $U_{\ell}:U_r$ and $U_{\ell}:U_r\in (M_1,\dots ,M_{t-1}):(M_t)$ as desired.
\medskip


\emph{Case 3:} Suppose that $x_n x_{n-1}$ divides $U_r$. Let $U_r=(x_{n-1}m_i)(x_nx_{n-2}l_j)$ be the maximal expression for some $m_i \in  \MR(P_{n-2})$ and $l_j \in  \MR(P_{n-3})$.
	
\underline{Claim 1:} $x_{n-1}\in (U_1,\dots ,U_{r-1}):(U_r)$.

\underline{Proof of Claim 1:} By Lemma~\ref{lem:every b contains some a} there exists $m_k \in \MR(P_{n-2})$ such that $m_k|x_{n-2}l_j$. Take $M=(x_{n-1}m_k)(x_{n-1}m_i) \in F(J(P_n)^2)$. If $M \in G(J(P_n)^2)$  then $M>_{\MR}U_r$ and $M:U_r=x_{n-1}$, which proves the claim. If $M \not\in G(J(P_n)^2)$, then by Lemma~\ref{weak} there exists $U_s \in  G(J(P_n)^2)$ such that $U_s \mid M$ and $U_s>_{\MR} M$. Thus $U_s>_{\MR}U_r$ and $U_s:U_r=x_{n-1}$, which proves our claim.

\underline{Claim 2:} If $i\geq 2$, we have $(m_1,\dots ,m_{i-1}):(m_i)\subseteq  (U_1,\dots ,U_{r-1}):(U_r)$.

\underline{Proof of Claim 2:} By Remark~\ref{rk:rooted order linear quotients} the ideal $(m_1,\dots ,m_{i-1}):(m_i)$ is generated by variables. To prove our claim, let $t<i$ such that $m_t:m_i=x_z$ for some variable $x_z$. Then consider $M=(x_{n-1}m_t)(x_nx_{n-2}l_j)$. If $M \in G(J(P_n)^2)$, then $M>_{\MR}U_r$ and $M:U_r=x_z$. Otherwise by Lemma~\ref{weak} there exists  $U_k \in  G(J(P_n)^2)$ such that $U_k \mid M$ and $U_k>_{\MR} M$.  Thus $U_k>_{\MR}U_r$ and $U_k:U_r=x_z$ which proves our claim. 

\underline{Claim 3:} If $j\geq 2$, then $(l_1,\dots ,l_{j-1}):l_j\subseteq  (U_1,\dots ,U_{r-1}):(U_r)$.

\underline{Proof of Claim 3:} By Remark~\ref{rk:rooted order linear quotients} the ideal $(l_1,\dots ,l_{j-1}):(l_j)$ is generated by variables.	Let $t<j$ such that $l_t:l_j=x_z$ for some variable $x_z$. Then consider $M=(x_{n-1}m_i)(x_n x_{n-2} l_t)$. If $M \in G(J(P_n)^2)$, then $M>_{\MR}U_r$ and $M:U_r=x_z$. Otherwise by Lemma~\ref{weak} there exists  $U_k \in  G(J(P_n)^2)$ such that $U_k \mid M$ and $U_k>_{\MR} M$.  Thus $U_k>_{\MR}U_r$ and $U_k:U_r=x_z$ which proves our claim. 

Let $t<r$. By Claim 1 and Claim 3, we may assume that $U_t=(x_{n-1}m_p)(x_nx_{n-2}l_q)$ and $p<i$. By Remark ~\ref{rk:rooted order linear quotients} there exists a variable $x_z\in (m_1,\dots, m_{i-1}):(m_{i})$ such that $x_z$ divides $m_p:m_i$. Observe that by Lemma~\ref{lem:properties of rooted list} (1) we have $x_z\neq x_{n-2}$. Proposition~\ref{prop:condition for non-minimal gen or non-maximal expression} implies the monomial $l_j$ is not divisible by $x_z$ as $U_r\in G(J(P_n)^2)$ and the expression $U_r=(x_{n-1}m_i)(x_nx_{n-2}l_j)$ is maximal. Thus $x_z$ divides $U_t:U_r$ and the result follows from Claim 2.
\end{proof}

\begin{Lemma}\label{lem:max degree of gen}
Let $a_n$ denote the maximum degree of a minimal monomial generator of $J(P_n)$. For any $n\geq 5$ we have
$$a_n=\max\{a_{n-2}+1, a_{n-3}+2\}. $$
For any $n\geq 2$
\begin{equation*}
a_n=
\begin{cases*}
2k & if $n=3k+1$ or $n=3k$ \\
2k+1        & if $n=3k+2$.
\end{cases*}
\end{equation*}
\end{Lemma}
\begin{proof}
The result follows from Lemma \ref{lem: rooted order for first power}.
\end{proof}

If an ideal has linear quotients, then its regularity is equal to the highest degree of a generator in a minimal set of generators, see \cite[Corollary~8.2.14]{herzog hibi monomial ideals}. Therefore as a consequence of Theorem~\ref{main thm} we obtain the following result.
\begin{Corollary} For any $n\geq 2$
\begin{equation*}
    \reg(J(P_n)^2)=
    \begin{cases*}
      4k & if $n=3k+1$ or $n=3k$ \\
      4k+2        & if $n=3k+2$.
    \end{cases*}
  \end{equation*}
  \end{Corollary}
\begin{proof}
	If $u\in G(J(P_n))$, then $u^2\in G(J(P_n)^2)$ by Lemma~\ref{lem:pure powers stay mingen}. The result follows from Lemma \ref{lem:max degree of gen}.
\end{proof}
\subsection{Concluding Remarks}

We can generalize the concept of rooted list to chordal graphs as follows. First, let us introduce some notation. If $v$ is a vertex of $G$, then the set of neighbors of $v$ is denoted by $N(v)$. The closed neighborhood of $v$ is $N[v]=N(v)\cup \{v\}$. If $A$ is a subset of vertices of $G$, then $G\setminus A$ denotes the graph which is obtained from $G$ by removing the vertices in $A$.

Suppose that $G$ is a chordal graph with a simplicial vertex $v_1$ such that $N[v_1]=\{v_1,\dots,v_r\}$ for some $r\geq 2$. Suppose that for each $i=1,\dots, r$, the list $\MR(H_i)$ is a rooted list of the subgraph $H_i=G\setminus N[v_i]$. Then we say 
$$ \MR(H_1)N(v_1),\MR(H_2)N(v_2),\dots ,\MR(H_r)N(v_r)$$
is a rooted list of $G$. Note that this list indeed consists of the minimal generators of $J(G)$, see \cite[Theorem~3.1]{erey}.

Observe that a path graph has only two simplicial vertices, namely the vertices at both ends of the path. However, a chordal graph in general can have many simplicial vertices. Therefore one can construct rooted lists of chordal graphs recursively in different ways. Below we give an example of how to construct a rooted list for a chordal graph. 

\begin{Example}
	Consider the graph in Figure~\ref{graph}. Observe that $a$ is a simplicial vertex with $N(a)=\{b,c\}$. We use the following rooted lists:
	$$\MR(G\setminus N[a])= de, ef, df \quad \MR(G\setminus N[b])=\emptyset \quad \MR(G\setminus N[c])=d,f. $$
	
	Also since we have $N(a)=bc$, $N(b)=acde$ and $N(c)=abe$ we can get the following rooted list for $G$:
	$$u_1, u_2, u_3, u_4, u_5, u_6:= bcde, bcef, bcdf, acde, abed, abef. $$
	Notice that since $J(G)$ is generated in single degree, every $2$-fold product $u_iu_j$ is a minimal generator of $J(G)^2$. There is only one minimal generator of $J(G)^2$ which has multiple expressions, namely
	$$ab^2cde^2f=u_1u_6=u_2u_5. $$
	Using Macaulay2 we listed the minimal generators of $J(G)^2$ in the rooted order as in Definition~\ref{def:rooted for second power} and we confirmed that such order yields linear quotients.
\end{Example}
\begin{figure}[hbt!]
	\centering
	\begin{tikzpicture}
	[scale=1.00, vertices/.style={draw, fill=black, circle, minimum size = 4pt, inner sep=0.5pt}, another/.style={draw, fill=black, circle, minimum size = 2.5pt, inner sep=0.1pt}]
	\node[another, label=left:{$a$}] (a) at (-2,0) {};
	\node[another, label=above:{$b$}] (b) at (-1, 1) {};
	\node[another, label=below:{$c$}] (c) at (-1, -1) {};
	\node[another, label=below:{$e$}] (e) at (1,-1) {};
	\node[another, label=above:{$d$}] (d) at (1,1) {};
	\node[another, label=above:{$f$}] (f) at (2.5, 0) {};
	\foreach \to/\from in {a/b, a/c, b/c, b/e, c/e, b/d, d/e, d/f, e/f}
	\draw [-] (\to)--(\from);
	\end{tikzpicture}
	\caption{\label{graph} A chordal graph}
\end{figure}
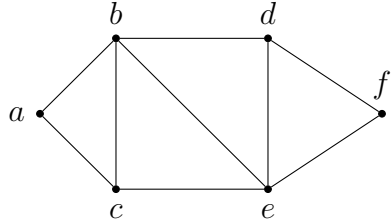
It would be interesting to know if the following question has a positive answer because it would settle the case of Conjecture~\ref{conjecture} for second powers.

\begin{question}
	If $G$ is a chordal graph, then does $J(G)^2$ has linear quotients with respect to a rooted list of minimal generators?
\end{question}

\section*{Acknowledgment}
The first author's research was supported by T\"{U}B\.{I}TAK, grant no. 118C033.

\end{document}